\documentclass{amsart}
\usepackage{color}
\usepackage{amssymb}
\usepackage{amsmath}
\usepackage{amscd}
\usepackage{amsbsy}
\usepackage{euscript}
\usepackage{verbatim}

\newtheorem{theorem}{Theorem}[section]
\newtheorem{lemma}[theorem]{Lemma}
\newtheorem{corollary}[theorem]{Corollary}
\newtheorem{proposition}[theorem]{Proposition}

\theoremstyle{definition}
\newtheorem{definition}[theorem]{Definition}
\newtheorem{remark}[theorem]{Remark}
\newtheorem{example}[theorem]{Example}

\newtheorem{def-not}[theorem]{Definition and Notation}

\renewcommand{\qed}{\hfill $\square$}

\definecolor{darkgreen}{rgb}{0.2, 0.7, 0.03}

\newcommand{\mx}{\rm{Max}}
\newcommand{\tmx}{\text{$t$-Max}}

\newcommand{\st}{\star}

\newcommand{\ra}{\Rightarrow}

\newcommand{\hgt}{{\rm ht}}
\newcommand{\tm}{\rm \text{$t$-} Max}
\newcommand{\bs}{\bigskip}

\newcommand{\mc}{\mathcal}

\begin{document}

\title{$\st$-super potent domains}

\author{Evan Houston}
\address{Department of Mathematics and Statistics, University of North Carolina at Charlotte,
Charlotte, NC 28223 U.S.A.}
\email{eghousto@uncc.edu}

\author{Muhammad Zafrullah}
\address{Department of Mathematic, Idaho State University, Pocatello, ID 83209 U.S.A.}
\email{mzafrullah@usa.net}


\thanks{The first-named author was supported by a grant from the Simons Foundation (\#354565)}

\date{\today}

\begin{abstract} For a finite-type star operation $\st$ on a domain $R$, we say that $R$ is $\st$-super potent if each maximal $\st$-ideal of $R$ contains a finitely generated ideal $I$ such that (1) $I$ is contained in no other maximal $\st$-ideal of $R$ and (2) $J$ is $\st$-invertible for every finitely generated ideal $J \supseteq I$.  Examples of $t$-super potent domains include domains each of whose maximal $t$-ideals is $t$-invertible (e.g., Krull domains).  We show that if the domain $R$ is $\st$-super potent for some finite-type star operation $\st$, then $R$ is $t$-super potent, we study $t$-super potency in polynomial rings and pullbacks, and we prove that a domain $R$ is a generalized Krull domain if and only if it is $t$-super potent and has $t$-dimension one.
\end{abstract} 

\maketitle


\section*{Introduction}

Dedekind domains are characterized as those domains having all nonzero ideals invertible.  On the other hand, if $D$ is a Dedekind domain with quotient field $k$ and $x$ is an indeterminate, then the domain $R:=D+(x^2,x^3)k[[x^2,x^3]]$ has invertibility strictly above $M:= (x^2,x^3)R$, but $M$ itself is not invertible in $R$.  Similarly, it is well known that Krull domains are characterized as those domains having all nonzero ideals $t$-invertible (definitions reviewed below), while in the example above, one has $t$-invertibility only above a certain level.  The goal of this paper is to explore one form of this kind of ($t$)-invertibility.

Now the $t$-operation is a particular example of a star operation, and it is useful to generalize to arbitrary finite-type star operations.  Let $R$ be a domain with quotient field $K$. Denoting by $\mc F(R)$ the set of nonzero fractional ideals of $R$, a map $\st:\mc F(R) \to \mc F(R)$ is a \emph{star operation on $R$} if the following conditions hold for all $A, B \in \mc F(R)$  and all $c \in K \setminus (0)$: \begin{enumerate}
\item $(cA)^{\st}=cA^{\st}$ and $R^{\st}=R$; 
\item $A \subseteq A^{\st}$, and, if $A \subseteq B$, then $A^{\st} \subseteq B^{\st}$; and 
\item $A^{\st \st}=A^{\st}$.
\end{enumerate}

An ideal $I$ satisfying $I^{\st}=I$ is called a $\st$-ideal. Other than the $d$-operation ($I^d=I$ for all nonzero fractional ideals $I$), the best known star operation is the \emph{$v$-operation}: for $I \in \mc F(R)$, put $I^{-1}=\{x \in K \mid xI \subseteq R\}$ and $I^v=(I^{-1})^{-1}$.   For any star operation $\st$, we may define an associated star operation $\st_f$ merely by setting, for $I \in \mc F(R)$, $I^{\st_f}=\bigcup  J^{\st}$, where the union is taken over all finitely generated subideals $J$ of $I$, and we say that $\st$ has \emph{finite type} if $\st=\st_f$.  The \emph{$t$-operation} is then given by $t=v_f$. It is well known that for a finite-type star operation $\st$ on a domain $R$, each $\st$-ideal is contained in a \emph{maximal} $\st$-ideal, that is, a star ideal maximal in the set of $\st$-ideals; maximal $\st$-ideals are prime; and $R=\bigcap R_P$, where the intersection is taken over the set of maximal $\st$-ideals $P$.  A nonzero ideal $I$ is \emph{$\st$-invertible} if $(II^{-1})^{\st}=R$.  For star operations $\st_1,\st_2$, we say that $\st_1 \le \st_2$ if $I^{\st_1} \subseteq I^{\st_2}$ for each $I \in \mc F(R)$.

Generalizing notions from \cite{aaz,acz}, we call a nonzero finitely generated ideal $I$ \emph{$\st$-rigid} if it is contained in a unique maximal $\st$-ideal and \emph{$\st$-super rigid} if, in addition, $J$ is $\st$-invertible for each finitely generated ideal $J \supseteq I$.  We say that a maximal $\st$-ideal $M$ is \emph{$\st$-potent} (\emph{$\st$-super potent}) if $M$ contains a $\st$-rigid ($\st$-super rigid) ideal and that the domain $R$ is \emph{$\st$-potent} (\emph{$\st$-super potent}) if each maximal $\st$-ideal of $R$ is $\st$-potent ($\st$-super potent).  It is clear that any domain each of whose maximal ideals is invertible is $d$-super potent, as is any valuation domain. On the other hand, a Krull domain may not (even) be $d$-potent (e.g., a polynomial ring in two indeterminates over a field) but is $t$-super potent.

For the remainder of the introduction, we assume that \emph{all star operations mentioned have finite type}.  In Section~\ref{s:super} we lay out many of the basic properties of $\st$-super potency.  In Corollary~\ref{c:tsp} we show that if $\st_1 \le \st_2$ and $R$ is $\st_1$-super potent, then it is also $\st_2$-super potent; in particular, since $d \le \st \le t$ for all (finite-type) $\st$, we have that $t$-super potency is the weakest type of super potency.  In Theorem~\ref{t:spchar} we obtain a local characterization: $R$ is $\st$-super potent if and only if $R$ is $\st$-potent and $R_M$ is $d$-super potent for each maximal $\st$ ideal $M$ of $R$.  In Theorem~\ref{t:observations} we establish, among other things, that if $I$ is a $\st$-super rigid ideal, then $\bigcap_{n=1}^{\infty} (I^n)^{\st}$ is prime. In Section~\ref{s:local} we study local super potency and show that a (non-field) local domain $(R,M)$ is $d$-super potent if and only there is a prime ideal $P \subsetneq M$ for which $P=PR_P$ and $R/P$ is a valuation domain. In a brief Section~\ref{s:poly}, we show that $t$-potency and $t$-super potency extend from $R$ to the polynomial ring $R[X]$.  Section~\ref{s:pullbacks} is devoted to determining how $t$-potency and $t$-super potency behave in a commonly studied type of pullback diagram, and these results are used to provide several examples.  In Section~\ref{s:tdim1} the main result is a characterization of Ribenboim's \emph{generalized Krull domains} \cite{ri}, those domains that may be expressed as a locally finite intersection of essential rank-one valuation domains: the domain $R$ is a generalized Krull domain if and only if it is $t$-super potent and every maximal $t$-ideal of $R$ has height one.


\section{Basic results on $\st$-super potency} \label{s:super}

From now on, we use $R$ to denote a domain and $K$ to denote its quotient field.  We begin by repeating the definition of $\st$-(super) potency.

\begin{definition} \label{d:(super) rigid} Let $\st$ be a finite-type star operation on the domain $R$.  Call a finitely generated ideal $I$ of $R$ \emph{$\st$-rigid} if it is contained in exactly one maximal $\st$-ideal of $R$ and \emph{$\st$-super rigid} if, in addition, each finitely generated ideal $J \supseteq I$ is $\st$-invertible.  We then say that a maximal $\st$-ideal of $R$ is \emph{$\st$-potent} (\emph{$\st$-super potent}) if it contains a $\st$-rigid ($\st$-super rigid) ideal and that $R$ itself is \emph{$\st$-potent} (\emph{$\st$-super potent}) if each maximal $\st$-ideal of $R$ is $\st$-potent ($\st$-super potent).
\end{definition}

\begin{remark} Recall that for a star operation $\st$ on $R$, a $\st$-ideal $A$ is said to have \emph{finite type} if $A=B^{\st}$ for some finitely generated ideal $B$ of $R$.  In \cite{aaz} a finite type $t$-ideal $J$ was dubbed \emph{rigid} if it is contained in exactly one maximal $t$-ideal.  For such a $J$, we have $J=I^t$ for some finitely generated subideal $I$ of $J$, and, since it is more convenient to work with the subideal $I$, we apply the ``rigid'' terminology to $I$ instead of $J$.  Moreover,  we want to consider finite-type star operations other than the $t$-operation (e.g., the $d$-operation!), and therefore prefer ``$t$-rigid'' in place of ``rigid.''  Similarly, we replace ``potent'' with ``$t$-potent.''
\end{remark}

In \cite{wm,wm2} Wang and McCasland studied the $w$-operation in the context of strong Mori domains.  Motivated by this, Anderson and Cook associated to any star operation $\st$ on a domain $R$ a finite-type star operation $\st_w$, given by $A^{\st_w}= \{x \in K \mid xB \subseteq A$ \text{for some finitely generated ideal B of}  $R$ with $B^{\st}=R\}$ \cite{ac}.  We always have $A^{\st_w}=\bigcap \{AR_P \mid P \in \st_f \text{-Max}(R)\}$ \cite[Corollary 2.10]{ac}, from which it follows that $A^{\st_w}R_P=AR_P$ for each $P \in \st_f \text{-Max}(R)$. (Recall that $\st_f$ is the finite-type star operation associated to $\st$ given by $A^{\st_f}=\bigcup B^{\st}$, where the union is taken over all finitely generated subideals $B$ of $A$.)  We also have $\st_f \text{-Max}(R) = \st_w \text{-Max}(R)$ \cite[Theorem 2.16]{ac}.  For the ``original'' $w$-operation, we have $w=v_w=t_w$ (hence the notation $\st_w$).

The following is an easy consequence of the definitions.

\begin{proposition} \label{p:2star}  Let $\st_1,\st_2$ be finite-type star operations on a domain $R$ for which $\st_1$-{\rm Max($R$)} $=$ $\st_2$-{\rm Max($R$)}.  Then $\st_1$-rigidity {\rm (}$\st_1$-potency, $\st_1$-super potency{\rm )} coincides with $\st_2$-rigidity {\rm (}$\st_2$-potency, $\st_2$-super potency{\rm )}.  In particular, $R$ is $\st_1$-potent {\rm (}$\st_1$-super potent{\rm)} if and only if $R$ is $\st_2$-potent {\rm (}$\st_2$-super potent{\rm )}. \qed
\end{proposition}

Observe that Proposition~\ref{p:2star} may be applied to $\st$ and $\st_w$ for any finite-type star operation $\st$ on a domain $R$. In particular, the proposition may be applied to $t$- and $w$-operations.

\begin{proposition} \label{p:increase} Let $\st_1 \le \st_2$ be finite-type star operations on a domain $R$.  If $M \in \st_1$-{\rm Max($R$)} $\cap$ $\st_2$-{\rm Max($R$)} and $M$ is $\st_1$-potent, then $M$ is $\st_2$-potent.
\end{proposition}
\begin{proof} Let $M \in \st_1$-{\rm Max($R$)} $\cap$ $\st_2$-{\rm Max($R$)} with $M$ $\st_1$-potent, and let $I$ be a $\st_1$-rigid ideal contained in $M$. Suppose that $I \subseteq N$ for some maximal $\st_2$-ideal $N$ of $R$.  Since $N^{\st_1} \subseteq N^{\st_2} \ne R$, there is a maximal $\st_1$-ideal $N'$ of $R$ for which $N \subseteq N'$.  Since $I \subseteq N'$, this forces $N'=M$ and hence $N=M$.  Therefore, $I$ is also $\st_2$-rigid, and hence $M$ is $\st_2$-potent.
\end{proof}

With respect to Proposition~\ref{p:increase}, it is not true that for finite-type star operations $\st_1 \le \st_2$ on a domain $R$ and $M \in \st_1$-{\rm Max($R$)} $\cap$ $\st_2$-{\rm Max($R$)} with $M$ $\st_2$-potent, we must also have $M$ $\st_1$-potent--see Example~\ref{e:increase} below. 
It is also not the case that for $\st_1 \le \st_2$ and $M$ a $\st_1$-potent maximal $\st_1$-ideal, we must have that $M$ is a maximal $\st_2$-ideal.  (Let $R=k[x,y]$, a polynomial ring in two variables over a field.  Then $M=(x,y)$ is a $d$-potent maximal ($d$)-ideal but is not a $t$-ideal.) More interestingly, it is not the case that, for $\st_1 \le \st_2$, $R$ $\st_1$-potent implies $R$ $\st_2$-potent, as we show in Example~\ref{e:not potent} below.

The situation is better for super potency:

\begin{theorem} \label{t:increase2} Let $\st_1 \le \st_2$ be finite-type star operations on a domain $R$.  If $M$ is  a $\st_1$-super potent maximal $\st_1$-ideal of $R$, then $M$ is also a $\st_2$-super potent maximal $\st_2$-ideal of $R$.
\end{theorem}
\begin{proof} Let $M$ be a $\st_1$-super potent maximal $\st_1$-ideal of $R$, and let $A \subseteq M$ be a $\st_1$-super rigid ideal.  We first show that $M^{\st_2} \ne R$.  If, on the contrary, $M^{\st_2}=R$, then there is a finitely generated ideal  $B \subseteq M$ with $B^{\st_2}=R$. Let $C:=A+B$.  Then $C$ is $\st_1$-invertible, whence $(C^{\st_2}C^{-1})^{\st_2}=(CC^{-1 })^{\st_2} \supseteq (CC^{-1})^{\st_1}=R$.  Since $C^{\st_2}=R$, this yields $C^{-1} = (C^{-1})^{\st_2} = R$.  However, the equation $(CC^{-1})^{\st_1}=R$ then forces $C^{\st_1}=R$, the desired contradiction.  Thus $M^{\st_2} \ne R$ and, since $M^{\st_2}$ is a $\st_1$-ideal, we must have $M^{\st_2}=M$.  Then, again since $\st_2$ ideals are also $\st_1$-ideals, it must be the case that $M$ is a maximal $\st_2$-ideal. That $M$ must be $\st_2$-super potent now follows easily, since for any finitely generated ideal $I \supseteq A$, $\st_1$-invertibility of $I$ implies $\st_2$-invertiblity.
\end{proof}
%

As a consequence of the preceding result, we have that the weakest type of super potency is $t$-super potency:

\begin{corollary} \label{c:tsp} Let $\st$ be a finite-type star operation on a domain $R$. \begin{enumerate}
\item If $M$ is a $\st$-super potent maximal $\st$-ideal of $R$, then $M$ is a $t$-super potent maximal $t$-ideal of $R$.
\item If $R$ is $\st$-super potent, then $R$ is $t$-super potent. \qed
\end{enumerate}
\end{corollary}

The converse of Corollary~\ref{c:tsp}(2) is false: if $k$ is a field, then the polynomial ring $k[X,Y]$, being a Krull domain, is $t$-super potent but is not $d$-super potent.  However, we do not know whether one can have a maximal ideal $M$ of a domain such that $M$ is a $t$-super potent maximal $t$-ideal but is not $d$-super potent.

Now let $R$ be a domain and $T$ a flat overring of $R$.  According to \cite[Proposition 3.3]{s}, if $\st$ is a finite-type star operation on $R$, then the map $\st_T: IT \mapsto I^{\st}T$ is a well-defined finite-type star operation on $T$.    
In the following result, we study how (super) potency extends to flat overrings.  We assume standard facts about flat overrings (including the fact, used above, that each fractional ideal of $T$ is extended from a fractional ideal of $R$); these follow readily from \cite{r}.

\begin{lemma} \label{l:flat} Let $R$ be a domain, $T$ a flat overring of $R$, $\st$ a finite-type star operation on $R$, and $\mathcal P$ the set of $\st$-primes $P$ of $R$ maximal with respect to the property $PT \ne T$.  Then:
\begin{enumerate}
\item $\st_T$-{\rm Max($T$)}  $=$ $\{PT \mid P \in \mathcal P\}$.
\item If $M$ is a $\st$-(super) potent maximal $\st$-ideal of $R$ for which $MT \ne T$, then $MT$ is a $\st_T$-(super) potent maximal $\st_T$-ideal of $T$.  (In fact, if $M$ is as hypothesized and $I \subseteq M$ is $\st$-(super) rigid, then $IT$ is $\st_T$-(super) rigid in $T$.)
\end{enumerate}
\end{lemma}
\begin{proof}  Let $P \in \mc P$.  Then $(PT)^{\st_T} = P^{\st}T=PT$, that is, $PT$ is a $\st_T$-ideal of $T$.  Moreover, if $Q$ is a prime of $R$ for which $QT$ is a maximal $\st_T$-ideal of $T$ containing $PT$, then $Q^{\st} \subseteq Q^{\st}T \cap R= (QT)^{\st_T} \cap R=QT \cap R=Q$; that is, $Q$ is a $\st$-ideal of $R$ containing $P$.  Since $P \in \mc P$, we have ($Q=P$ and hence) $QT=PT$.  Therefore $PT$ is a maximal $\st_T$-ideal of $T$.  Conversely, let $P$ be a prime of $R$ for which $PT$ is a maximal $\st_T$-ideal of $T$.  Then $P^{\st} \subseteq P^{\st}T \cap R = (PT)^{\st_T} \cap R=PT \cap R=P$, and so $P$ is a $\st$-ideal of $R$.  Suppose that $P \subseteq Q$, where $Q$ is a $\st$-prime of $R$ and $QT \ne T$.  Then $QT$ is a $\st_T$-ideal of $T$ (since, $(QT)^{\st_T}=Q^{\st}T=QT$) containing $PT$, whence ($QT=PT$ and hence) $Q=P$.  This proves (1).

Let $M$ be a $\st$-potent maximal $\st$-ideal of $R$ such that $MT \ne T$.  Then $MT$ is a maximal $\st_T$-ideal of $T$ by (1).  Now let $I$ be a $\st$-rigid ideal contained in $M$, and suppose that $IT \subseteq NT$, where $N$ is a prime ideal of $R$ for which $NT$ is a maximal $\st_T$-ideal of $T$.  Then $N^{\st} \ne R$, whence $N \subseteq N'$ for some maximal $\st$-ideal $N'$ of $R$. Since $I$ is contained in no maximal $\st$-ideal of $R$ other than $M$, we must have $N'=M$. However, this yields $N \subseteq M$ and hence $NT=MT$.  It follows that $IT$ is $\st_T$-rigid in $T$.

Now assume that $M$ is $\st$-super potent and that $I \subseteq M$ is $\st$-super rigid.  Let $J$ be a finitely generated ideal of $R$ for which $JT \supseteq IT$.  Replacing $J$ with $I+J$ if necessary, we may assume that $J \supseteq I$.  Then $J$ is $\st$-invertible, whence, in particular, $JJ^{-1} \nsubseteq M$.  This, in turn,  yields $(JT)(T:JT) \nsubseteq MT$.  Since $MT$ is the only maximal $\st_T$-ideal of $T$ containing $JT$, $JT$ is $\st$-invertible.  Therefore, $IT$ is $\st_T$-super rigid.
This completes the proof of (2).
\end{proof}

\begin{remark} \label{r:principal} Suppose that $(R,M)$ is local and that $\st$ is a star operation on $R$ for which $M$ is a $\st$-ideal. Then if $I$ is a $\st$-invertible ideal of $R$, we cannot have $II^{-1} \subseteq M$, and hence $I$ is actually (invertible and hence) principal.  In particular, if $\st$ is of finite-type and $I \subseteq M$ is $\st$-super rigid, then $I$ is principal.  We shall use this fact often in the sequel.
\end{remark}

\begin{lemma} \label{l:locsp} Let $(R,M)$ be a local domain.  The following statements are equivalent. \begin{enumerate}
\item $M$ is a $\st$-super potent maximal $\st$-ideal for some finite-type star operation $\st$ on $R$.
\item $M$ is a $t$-super potent maximal $t$-ideal.
\item $M$ is $d$-super potent.
\item $M$ is a $\st$-super potent maximal $\st$-ideal for every finite-type star operation on $R$.
\end{enumerate}
\end{lemma}
\begin{proof}  The implications (1) $\ra$ (2) and (3) $\ra$ (4) follow from Theorem~\ref{t:increase2}, and (4) $\ra$ (1) is trivial.  Assume (2), and let $A$ be a $t$-super rigid ideal contained in $M$ and $B$ a finitely generated ideal containing $A$. Then $B$ is $t$-invertible and hence principal (Remark~\ref{r:principal}).  Therefore, $A$ is $d$-super rigid, as desired.
\end{proof}

Since the extension (as defined above) of the $d$-operation on $R$ to a flat overring $T$ is the $d$-operation on $T$, we shall write ``$d$'' instead of ``$d_T$'' in this case.

It is now an easy matter to characterize $\st$-super potency locally:

\begin{theorem} \label{t:spchar} Let $\st$ be a finite-type star operation on $R$ and $M$ a maximal $\st$-ideal of $R$.  Then: \begin{enumerate}
\item $M$ is $\st$-super potent if and only if $M$ is $\st$-potent and $MR_M$ is $d$-super potent.
\item $R$ is $\st$-super potent if and only if $R$ is $\st$-potent and $R_P$ is $d$-super potent for each maximal $\st$-ideal $P$ of $R$.
\end{enumerate}
\end{theorem}
\begin{proof} It suffices to prove (1). Assume that $M$ is $\st$-potent and $MR_M$ is $d$-super potent.  Then there is a finitely generated ideal $A$ of $R$ such that $AR_M$ is $d$-super rigid and a $\st$-rigid ideal $B$ of $R$ contained in $M$. Suppose that $C$ is a finitely generated ideal of $R$ with $C \supseteq A+B$.  Then $CR_M$ is principal, and $CR_N=R_N$ for each maximal $\st$-ideal $N \ne M$.  It follows that $C$ is $\st$-invertible.  Thus ($A+B$ is $\st$-super rigid and hence) $M$ is $\st$-super potent.  For the converse, if $M$ is $\st$-super potent, then $M$ is certainly $\st$-potent.  Moreover, $MR_M$ is $\st_{R_M}$-super potent by Lemma~\ref{l:flat} and hence $d$-super potent by Lemma \ref{l:locsp}.
\end{proof}

In spite of Theorem~\ref{t:spchar} (and Lemma~\ref{l:flat}), $\st$-super potency does not in general localize at non-maximal $\st$-primes--see Example~\ref{e:loc} below. Also, observe that if $R$ is a non-Dedekind almost Dedekind domain, then $R_M$ is $d$-super potent for each maximal ($t$-)ideal $M$, but $R$ is not $t$-potent.  (Hence the $\st$-potency assumption is necessary in Theorem~\ref{t:spchar}(1,2).)

\begin{theorem} \label{t:observations} Let $\st$ be a finite-type star operation on a domain $R$, $M$ a $\st$-super potent maximal $\st$-ideal of $R$, and $I$ a $\st$-super rigid ideal of $R$ contained in $M$.  \begin{enumerate}
\item If $A$ is a finitely generated ideal for which $A^{\st} \supseteq I$, then $A$ is $\st$-super rigid.
\item If $J$ is a $\st$-super rigid ideal contained in $M$, then $I \subseteq J^{\st}$ or $J \subseteq I^{\st}$.
\item If $J$ is a $\st$-super rigid ideal contained in $M$, then $IJ$ is also a $\st$-super rigid ideal.
\item $I^n$ is $\st$-super rigid for each positive integer $n$.
\item If $R$ is local with maximal ideal $M$, then $I$ is comparable to each ideal of $R$, and $\bigcap_{n=1}^{\infty} I^n$ is prime. 
\item $I^{\st}=IR_M \cap R$.
\item $\bigcap_{n=1}^{\infty} (I^n)^{\st}$ is prime.
\item If $P$ is a prime ideal of $R$ with $P \subseteq M$ and $I \nsubseteq P$, then $P \subseteq \bigcap_{n=1}^{\infty} (I^n)^{\st}$.
\end{enumerate}
\end{theorem}
\begin{proof} (1) Let $A$ be a finitely generated ideal with $A^{\st} \supseteq I$.  Then $A$ is clearly $\st$-rigid.  Let  $B$ be a finitely generated ideal with $B \supseteq A$. Set $C:=I+B$.  Then $C$ is $\st$-invertible, and, since $C^{\st} = (I^{\st} +B^{\st})^{\st} \subseteq (A^{\st} +B^{\st})^{\st} = B^{\st}$, we have $C^{\st}= B^{\st}$, and hence $B$ is $\st$-invertible.  Therefore, $A$ is $\st$-super rigid.

(2) Let $J$ be a $\st$-super rigid ideal contained in $M$, and set $C:=I+J$.  Then $C$ is $\st$-invertible, and we have $(IC^{-1}+JC^{-1})^{\st}=R$.  Note that $IC^{-1} \supseteq I$ and $JC^{-1} \supseteq J$, and hence $IC^{-1} \nsubseteq M$ or $JC^{-1} \nsubseteq M$.  Since $IC^{-1}, JC^{-1}$ can be contained in no maximal $\st$-ideal of $R$ other than $M$, we must have $(IC^{-1})^{\st}=R$ or $(JC^{-1})^{\st}=R$, that is, $C^{\st}=I^{\st}$ or $C^{\st}=J^{\st}$. The conclusion follows easily.

(3) Again, let $J$ be a $\st$-super rigid ideal contained in $M$, and let $C$ be a finitely generated ideal containing $IJ$.    Since $I$ is $\st$-invertible, $I^{-1}=A^{\st}$ for some finitely generated ideal $A$. This yields $(CA)^{\st} \supseteq (IJA)^{\st} =J^{\st} \supseteq J$, and hence $CA$ is $\st$-invertible.  It follows that $C$ is $\st$-invertible.

(4) This follows from (3).

(5) Assume that $R$ is local with maximal ideal $M$.  By Lemma~\ref{l:locsp} (and its proof) $I$ is $d$-super rigid and therefore principal (Remark~\ref{r:principal}), say $I=(c)$. Choose $r \in M \setminus (c)$.  Then $(c,r)$ is principal, and, since $R$ is local, $(c,r)=(r)$, i.e. $c \in (r)$.  It follows that $I$ is comparable to each ideal of $R$. Now suppose, by way of contradiction, that $a,b \in R$ with $ab \in \bigcap (c^n)$ and $a,b \notin \bigcap (c^n)$. Choose $n,m$ with $a \in (c^n) \setminus (c^{n+1})$ and $b \in (c^m) \setminus (c^{m+1})$.  Then $a/c^n, b/c^m \notin (c)$, whence, by the claim, $c \in (a/c^n) \cap (b/c^m)$.  Hence $c^{n+m+2} \in (ab) \subseteq (c^{n+m+3})$, yielding the contradiction that $1 \in (c)$.  Hence $\bigcap_{n=1}^{\infty} I^n$ is prime.

(6) We have $I^{\st} \subseteq I^{\st}R_M \cap R = (IR_M)^{\st_{R_M}} \cap R = IR_M \cap R$ (since $IR_M$ is principal).  On the other hand, $IR_N=R_N$ for $N \in \mc N:=\st\text{-Max}(R) \setminus \{M\}$, and hence $I^{\st} \supseteq I^{\st_w} = IR_M \cap (\bigcap_{N \in \mc N} IR_N)=IR_M \cap R$.

(7) By (4), Lemma~\ref{l:flat}, and (the proof of) Lemma~\ref{l:locsp}, $I^nR_M$ is $d$-super rigid for each $n$.  Using (6), we have $\bigcap_{n=1}^{\infty} (I^n)^{\st}=\bigcap_{n=1}^{\infty} (I^nR_M \cap R)=(\bigcap_{n=1}^{\infty}I^nR_M) \cap R$, which is prime by (5).

(8) Let $P$ be as described.  Since $IR_M \nsubseteq PR_M$, we have by (5) and (6) that $P \subseteq \bigcap_{n=1}^{\infty} I^nR_M \cap R =\bigcap_{n=1}^{\infty} (I^n)^{\st}$.
\end{proof}

We record the following useful consequence of Theorem~\ref{t:observations}.

\begin{corollary} \label{c:locval} If $M$ is a $t$-super potent ideal of height one in a domain $R$, then $R_M$ is a valuation domain.  In particular, a one-dimensional local $d$-super potent domain is a valuation domain.
\end{corollary}
\begin{proof} We begin with the ``in particular'' statement. Let $R$ be a one-dimensional local $d$-super potent domain,  $I$ a $d$-super rigid ideal of $R$, and $J$ a finitely generated ideal of $R$. Then $J \supseteq I^n$ for some positive integer $n$.  Since $I^n$ is $d$-super rigid by Theorem~\ref{t:observations}, $J$ must be (invertible and hence) principal.  It follows that $R$ is a valuation domain.  Now assume that $M$ is $t$-super potent of height one in a domain $R$.  By Theorem~\ref{t:spchar}, $R_M$ is $d$-super potent and is therefore a valuation domain by what has just been proved.
\end{proof}

It is easy to see that the requirement on the height of $M$ in Corollary~\ref{c:locval} is necessary--take $R$ to be any local non-valuation domain having principal maximal ideal and dimension at least two.


\section{the local case} \label{s:local}

Let $(R,M)$ be a local domain and $\st$ a finite-type star operation on $R$.  Recall from Lemma~\ref{l:locsp} that $R$ is $\st$-super potent if and only if $R$ is $d$-super potent.  We shall characterize and study local $d$-super potency.

As in \cite{d} we say that a prime ideal $P$ of a domain $R$ is \emph{divided} if $P=PR_P$.  Domains in which each prime ideal is divided were introduced and  briefly studied in \cite{ak}, apparently motivated by considerations from \cite{go}.  
Recall that if $P$ is a prime ideal of a domain $R$, then $R+PR_P$ is called the \emph{CPI-extension of $R$ with respect to $P$} \cite{bs}. (``CPI'' is short for ``complete pre-image.'')  The next lemma follows easily from arguments in \cite{ak,d,bs}.

\begin{lemma} \label{l:divided} Let $P$ be a prime ideal of a domain $R$.  Then the following statements are equivalent. \begin{enumerate}
\item $P$ is divided.
\item $P$ is comparable to each principal ideal of $R$.
\item $P$ is comparable to each ideal of $R$.
\item $R$ is the CPI-extension of $R$ with respect to $P$.
\end{enumerate}
\end{lemma}

\begin{theorem} \label{t:val} Let $(R,M)$ be a local domain, not a field.  Then $R$ is  $d$-super potent if and only if there is a divided prime $P \subsetneq M$ such that $R/P$ is a valuation domain.
\end{theorem}
\begin{proof} Suppose that $R$ is $d$-super potent, and let $I \subseteq M$ be $d$-super rigid.  Then $I=(c)$ for some $c \in M$. Moreover, by Theorem~\ref{t:observations}(4,5), $P:= \bigcap (c^n)$ is prime, and, for each positive integer $m$, $(c^m)$ is $d$-super rigid and hence comparable to each ideal of $R$ .  Let $a \in M \setminus P$.  Then $a \notin (c^k)$ for some $k$, whence $P \subseteq (c^k) \subseteq (a)$.  Hence $(a)$ is $d$-super rigid.  This shows both that $P$ is divided (Lemma~\ref{l:divided}) and that any two principal ideals generated by elements  of $M \setminus P$ must be comparable (since each is a $d$-super rigid ideal).  It follows that $R/P$ is a valuation domain.

Now assume that $P$ is a divided prime properly contained in $M$ and that $R/P$ is a valuation domain.  Let $a \in M \setminus P$.  Since $P$ is divided, we have $P  \subsetneq (a)$ (Lemma~\ref{l:divided}). Suppose that $I=(a_1, \ldots, a_n)$ is a finitely generated ideal containing $(a)$.  Then $I/P \supseteq (a)/P$ in the valuation domain $R/P$, and it follows that ($I/P$ and hence) $I$ is principal. Therefore, $(a)$ is super rigid.  
\end{proof}

Recall from Corollary~\ref{c:locval} that a one-dimensional $d$-super potent domain is a valuation domain. Of course, this is also an immediate corollary of Theorem~\ref{t:val}, as is the following result in the two-dimensional case.

\begin{corollary} \label{c:dim2} If $R$ is a two-dimensional local $d$-super potent domain, then $R$ has exactly two nonzero prime ideals. \qed
\end{corollary}

It is trivial that a Noetherian domain $R$ is $\st$-potent for any star operation $\st$ on $R$. As another consequence of Theorem~\ref{t:val}, we have a characterization of Noetherian $t$-super potent domains:

\begin{corollary} \label{c:noe} Let $R$ be a Noetherian domain.  \begin{enumerate} 
\item If $M$ is a $t$-super potent maximal $t$-ideal of $R$, then $\text{\hgt}(M)=1$.
\item If $R$ is $t$-super potent, then $R$ is a Krull domain.
\end{enumerate}
\end{corollary}
\begin{proof} (1) Let $M$ be a $t$-super potent maximal $t$-ideal of $R$.  Then $R_M$ is a $d$-super potent Noetherian domain, and hence we may as well assume that $R$ is local with $d$-super potent maximal ideal $M$. By Theorem~\ref{t:val}, there is a divided prime $P \subsetneq M$ such that $R/P$ is a Noetherian valuation domain.  Moreover, if we choose $a \in M \setminus P$ and shrink $M$ to a prime $Q$ minimal over $a$, then $Q \supseteq (a) \supsetneq P$.  By the principal ideal theorem, we must have $\hgt(Q)=1$, and hence $P=(0)$.  But then $R$ is a Noetherian valuation domain, and we must have $\hgt(M)=1$.  For (2), suppose that $R$ is $t$-super potent.  By (1) $R_M$ is a Noetherian valuation domain for each $M \in t$-\text{Max}($R$), and hence the representation $R=\bigcap \{R_M \mid M \in \text{$t$-Max}(R)\}$ shows that $R$ is (completely) integrally closed and therefore a Krull domain.
\end{proof}

\begin{remark} \label{r:gmz} (1) Recall from \cite{gmz} that a nonzero element $a$ of a domain $R$ is said to be \emph{comparable} if $(a)$ compares to each ideal of $R$ under inclusion.  By \cite[Theorem 2.3]{gmz} and Theorem~\ref{t:val}, non-field local $d$-super potent domains coincide with domains that admit nonzero, nonunit comparable elements.  Moreover, again by \cite[Theorem 2.3]{gmz}, for such a domain $R$, the ideal $P_0:= \bigcap \{(c) \mid c \text{ is a nonzero comparable element of } R\}$ is a divided prime and is such that $R/P_0$ is a valuation domain, and $P_0$ is the (unique) smallest prime $L$ of $R$ such that $L$ is divided and $R/L$ is a valuation domain.

(2) With the notation above, the following statements are equivalent: (a) $R$ is a valuation domain, (b) $R_{P_0}$ is a valuation domain, and (c) $P_0 = (0)$: the implications (a) $\ra$ (b) and (c) $\ra$ (a) are clear, and (b) $\ra$ (c)  by the remark following Theorem 2.3 of \cite{gmz}.

(3) As explained in the just-mentioned remark in \cite{gmz}, every local domain $(R,M)$ that admits a nonzero, nonunit comparable element arises as a pullback 

$$  \begin{CD}
        R   @>>>    V\\
        @VVV        @VVV    \\
        T  @>\varphi>>   T/M = k,\vspace{.08in}
\end{CD}$$ 
where $(T,M)$ is a local domain and $V$ is a valuation domain with quotient field $k$ (in which case we have $T=R_M$). In particular, if $T$ is a two-dimensional Noetherian domain, it must have infinitely many height-one primes and hence so must $R$.  Thus Corollary~\ref{c:dim2} does not extend to higher dimensions; indeed, the primes of a local $d$-super potent domain need not even be linearly ordered (e.g. $R=\mathbb Z_{(p)}+(x,y)\mathbb Q[[x,y]]$, where $p$ is prime and $x, y$ are indeterminates).
\end{remark}

We end this section with an attempt to globalize local $\st$-super potency.


\begin{lemma} \label{l:zero} Let $M,N,P$ be primes in a domain $R$ with $P \subseteq M \cap N$, and assume that $PR_M$ is divided in $R_M$ and that $R_M/PR_M$ and $R_N$  are valuation domains.  Then $R_M$ is a valuation domain.
\end{lemma}
\begin{proof}  Since $R_N$ is a valuation domain, so is $R_P=(R_N)_{PR_N}$.  However, we also have $R_P=(R_M)_{PR_M}$, and since $R_M =\varphi^{-1}(R_M/PR_M)$, where $\varphi: R_P \to R_P/PR_P$ is the canonical projection, $R_M$ is a valuation domain by \cite[Proposition 18.2(3)]{g2}.
\end{proof}

\begin{definition} \label{d:belong} Let $R$ be a domain, and let $P \subsetneq M$ be prime ideals of $R$. We say that  $P$  \emph{belongs} to $M$ if $PR_M=PR_P$ and $R_M/PR_M$ is a valuation domain.
\end{definition}

Note that, by Theorem~\ref{t:val}, a prime ideal $M$ of a domain $R$ contains a belonging prime if and only if $R_M$ is $d$-super potent.  Moreover, if $M$ contains a belonging prime, then it contains a smallest one by Remark~\ref{r:gmz}(1).

\begin{lemma} \label{l:commonmin} Let $R$ be a domain, let $M, N$ be prime ideals of $R$, 
and suppose that there is a prime belonging to both $M$ and $N$.  Then the smallest prime of $R$ that belongs to $N$ also belongs to $M$ {\rm (}and vice versa{\rm )}.
\end{lemma}
\begin{proof} Let $P$ belong to both $M$ and $N$, and let $Q$ be the smallest prime belonging to $N$.  We have $Q \subseteq P$. Applying Lemma~\ref{l:zero} to $R/Q$ yields that $R_M/QR_M=(R/Q)_{M/Q}$ is a valuation domain.  Also, since $QR_N$ is divided in $R_N$, $$QR_Q=QR_N \subseteq QR_P \subseteq PR_P=PR_M \subseteq R_M.$$  Hence $QR_Q =QR_Q \cap R_M=QR_M$.  Therefore, $Q$ belongs to $M$.
\end{proof}

\begin{remark} \label{r:equiv} Let $\st$ be a finite-type star operation on a domain $R$, and assume that $R$ is $\st$-super potent.  Then each maximal $\st$-ideal of $R$ contains a belonging prime by Theorems~\ref{t:spchar} and \ref{t:val}.  Define $\sim$ on $\st$-{\rm Max($R$)} by $M \sim N$ if $M$ and $N$ contain a common belonging prime.  It is perhaps interesting that $\sim$ is an equivalence relation: it is clearly reflexive and symmetric, and transitivity follows easily from Lemma~\ref{l:commonmin}.
\end{remark}

Observe that the relation described above forces a certain amount of ``independence'' in $\st$-\text{Max}($R$): if $M, N$ are two maximal $\st$-ideals in the $\st$-super potent domain $R$ with $M \not \sim N$, $P$ belongs to $M$, $Q$ belongs to $N$, and $Q \subseteq P$, then $P \nsubseteq N$. We give a simple example illustrating this. \bs

\begin{example} \label{e:funny} Let $F$ be a field, and $x,y$ indeterminates.  Set $V=F(x)[y]_{yF(x)[y]}$, $T=F(y)[x^2,x^3]_{(x^2,x^3)F(y)[x^2,x^3]}$, $R_1=V+P$, where $P$ is the maximal ideal of $T$, $R_2=F(x)[y]_{(y+1)F(x)[y]}$, and $R=R_1 \cap R_2$.  Then $R_1$ and $R_2$ are $d$-super potent (both have principal maximal ideals).  Denote the maximal ideal of $R_1$ by $M_1$. Then $M:=M_1 \cap R$ and $N:=(y+1)R_2 \cap R$ are the maximal ideals of $R$, and by \cite[Theorem 3]{p}, we have $R_M=R_1$ and $R_N=R_2$.   The domain $R$ is therefore $d$-super potent by Theorem~\ref{t:spchar}, and it is clear that $(0)$ belongs to $N$ and that $P$ (but not $(0)$) belongs to $M$.
\end{example}


\section{Polynomial rings over $t$-super potent domains} \label{s:poly}

We begin with some well-known facts about $t$-ideals in polynomial rings. Recall that if $R$ is a domain and $Q$ is  a nonzero prime of $R[X]$ for which $Q \cap R = (0)$, then $Q$ is called an \emph{upper to zero}.

\begin{lemma} \label{l:polyfacts} Let $R$ be a domain. \begin{enumerate}
\item An ideal $A$ of $R$ is a $t$-ideal if and only if $A[X]$ is a $t$-ideal of $R[X]$.
\item If $Q$ is maximal $t$-ideal of $R[X]$, then $Q=P[X]$ for some maximal $t$-ideal of $R$ or $Q$ is an upper to zero in $R[X]$.
\item An ideal $M$ of $R$ is a maximal $t$-ideal if and only if $M[X]$ is a maximal $t$-ideal of $R[X]$.
\item If $Q$ is an upper to zero in $R[X]$ and is also a maximal $t$-ideal, then $Q$ is $t$-super potent.
\end{enumerate}
\end{lemma}
\begin{proof} For (1) see \cite[Proposition 4.3]{hh}.  Let $Q$ be a maximal $t$-ideal of $R[X]$.  By \cite[Proposition 1.1]{hz}, $Q=(Q \cap R)[X]$ or $Q$ is an upper to zero.  It then follows from (1), that if $Q=(Q \cap R)[X]$, then $P:=Q \cap R$ must be a maximal $t$-ideal of $R$. This gives (2), and (3) follows from (1) and (2). Now suppose that $Q$ is an upper to zero and also a maximal $t$-ideal in $R[X]$.  Then $Q=fK[X] \cap R[X]$ for some polynomial $f \in Q$ such that $f$ is irreducible in $K[X]$.  By \cite[Theorem 1.4]{hz} there is an element $g  \in Q$ such that $c(g)^v=R$ (where $c(g)$, the content of $g$, is the ideal of $R$ generated by the coefficients of $g$), and it is easy to see via (1) and (2) that the ideal $(f,g)$ of $R[X]$ is contained in no maximal $t$-ideal of $R[X]$ other than $Q$.  Hence $Q$ is $t$-potent and therefore by Theorem~\ref{t:spchar} also $t$-super potent since $R[X]_Q$ is a valuation domain.  Hence (4) holds.
\end{proof}

\begin{theorem} \label{t:poly} Let $R$ be a domain.  Then $R$ is $t$-(super) potent if and only if $R[X]$ is $t$-(super) potent.
\end{theorem}
\begin{proof} Suppose that $R$ is $t$-potent, and let $Q$ be a maximal $t$-ideal of $R[X]$.  By Lemma~\ref{l:polyfacts}(2), $Q$ is either an upper to zero or $Q=P[X]$ with $P$ a maximal $t$-ideal of $R$.  If $Q$ is an upper to zero, it is $t$-super potent by Lemma~\ref{l:polyfacts}(4).  If $Q=P[X]$ with $P \in \tmx(R)$, then there is a $t$-rigid ideal $I$ of $R$ contained in $P$, and it is easy to see that $I[X]$ is $t$-rigid in $R[X]$.  Hence $R[X]$ is $t$-potent. 

Now assume that $R$ is $t$-super potent.  Then $R[X]$ is $t$-potent by what has already been proved.  Hence, by Theorem~\ref{t:spchar}, it suffices to show that $R[X]_Q$ is $d$-super potent for each maximal $t$-ideal $Q$ of $R[X]$.  To this end, let $Q$ be a maximal $t$-ideal of $R[X]$. Again by Lemma~\ref{l:polyfacts}(4), we may as well assume that $Q=P[X]$ with $P$ a maximal $t$-ideal of $R$.  We shall show that $R[X]_Q$ satisfies the requirements of Theorem~\ref{t:val}.  Since $R[X]_Q=R_P[X]_{PR_P[X]}$ and $R_P$ is $d$-super potent, we change notation and assume that $R$ is local with $d$-super potent maximal ideal $P$, and we wish to show that $R[X]_{P[X]}$ is $d$-super potent. By Theorem~\ref{t:val} there is a prime $L$ of $R$ such that $L \subsetneq P$, $R/L$ is a valuation domain, and $L=LR_L$.  Then $R[X]_{P[X]}/LR[X]_{P[X]} =(R/L)[X]_{(P/L)[X]}$, which is a  valuation domain.  Finally, we must show that $LR[X]_{L[X]}=LR[X]_{P[X]}$.  Let $f,g \in R[X]$ with $c(g) \subseteq L$ and $f \in R[X] \setminus L[X]$.  If $f \notin P[X]$, then $g/f \in LR[X]_{P[X]}$, as desired.  Suppose that $f \in P[X]$.  Since $L=LR_P$ and $f \notin L[X]$, $c(f) \supsetneq c(g)$, and, since $R/L$ is a valuation domain, $c(f)=(b)$ for some $b \in P \setminus L$.  Note that $b^{-1}f \in R[X] \setminus P[X]$.  Also, since $b^{-1}g\cdot b \in L[X]$ and $b \notin L$, $b^{-1}g \in L[X]$.  Thus $g/f=b^{-1}g/(b^{-1}f) \in LR[X]_{P[X]}$, as desired.

For the converse, first assume that $R[X]$ is $t$-potent, and let $P$ be a maximal $t$-ideal of $R$.  Then $P[X]$ is a maximal $t$-ideal of $R[X]$, and we may find a $t$-rigid ideal $A \subseteq P[X]$.  Let $I$ denote the ideal of $R$ generated by the coefficients of the polynomials in a finite generating set of $A$.  Then $I$ is a finitely generated ideal of $R$ contained in $P$, and since $A \subseteq I[X] \subseteq P[X]$ yields that $I[X]$ is $t$-rigid in $R[X]$, it is clear that $I$ is $t$-rigid in $R$.  Hence $R$ is $t$-potent.  Finally, suppose that $R[X]$ is $t$-super potent.  Using the notation above, we may assume that $A$ is $t$-super rigid, whence $I[X]$ is also $t$-super rigid.  If $J$ is a finitely generated ideal of $R$ containing $I$, then $J[X]$ is a finitely generated ideal of $R[X]$ containing $I[X]$; this yields that $J[X]$ is $t$-invertible in $R[X]$, from which it follows easily that $J$ is $t$-invertible in $R$.  Hence $t$-super potency of $R[X]$ implies $t$-super potency of $R$.
\end{proof}

\begin{remark} \label{r:proof} It is interesting to note that in the proof above, it was easy to show that $R[X]_{PR[X]}/LR[X]_{P[X]}$ is a valuation domain using only the fact that $R/L$ is a valuation domain, but the proof that $LR[X]_{L[X]}=LR[X]_{P[X]}$ used not only the assumption that $L=LR_L$ but also the assumption that $R/L$ is a valuation domain.  Here is an example that shows the necessity of the latter assumption.  Let $F$ be a field, $k=F(u)$, $u$ an indeterminate, $V$ a 2-dimensional valuation domain of the form $k+P$ with height-one prime $L$, and $R=F+P$.  According to \cite[Theorem 19.15 and its proof]{g2}, denoting the common quotient field of $R$ and $V$ by $K$,  $Q:= (X-u)K[X] \cap R[X]$ is an upper to zero in $R[X]$ satisfying $Q \subseteq P[X]$.  We have $L=LV_L=LR_L$.  However, $R/L$ is not a valuation domain, and we claim that we do not have $LR[X]_{P[X]}=LR[X]_{L[X]}$.  To see this choose $a \in L$, $a \ne 0$, and $c \in P \setminus L$.  Then $a/(cX-cu) \in LR[X]_{L[X]}$.  Suppose that we can write $a/(cX-cu) = g/f$ with $g \in L[X]$ and $f \in R[X] \setminus P[X]$. We have $af=g(cX-cu)$, so that $f =a^{-1}g(cX-cu ) \in (X-u)K[X] \cap R[X] = Q \subseteq P[X]$, a contradiction.  This verifies the claim.
\end{remark}


\section{Pullbacks} \label{s:pullbacks}

Let $T$ be a domain, $M$ a maximal ideal of $T$, $\varphi: T \to k := T/M $ the natural projection, and $D$ a proper subring of $k$.  Then let $R = \varphi^{-1}(D)$ be the integral domain arising from the following pullback of canonical homomorphisms.

  $$  \begin{CD}
        R   @>>>    D\\
        @VVV        @VVV    \\
        T  @>\varphi>>   T/M = k.\\
\end{CD}$$

We list some properties that we shall need.

\begin{lemma} \label{l:tmaxpullback} Consider the pullback diagram above. \begin{enumerate}
\item $T$ is a flat $R$-module if and only if $k$ is the quotient field of $D$.
\item If $I$ is a nonzero finitely generated ideal of $D$, then $\varphi^{-1}(I)$ is a finitely generated ideal of $R$.
\item $t$-{\rm Max($R$)} $=$  $\{N \cap R \mid N \in$ $t$-{\rm Max($T$)}, $N \nsubseteq M\} \bigcup \{\varphi^{-1}(P') \mid P' \in$ $t$-{\rm Max($D$)}\}.  {\rm (}By convention, if $D$ is a field, then $(0)$ is a maximal $t$-ideal of $D$, in which case $M$ is a maximal $t$-ideal of $R${\rm )}.
\item If $N$ is a prime ideal of $T$ that is incomparable to $M$, then $R_{N\cap R}=T_N$.
\item Assume that $D$ is not a field.  If $I$ is a $t$-invertible ideal of $R$ with $I \supsetneq M$, then $\varphi(I)$ is a $t$-invertible ideal of $D$.  Conversely, if $I'$ is a $t$-invertible ideal of $D$, then $\varphi^{-1}(I')$ is a $t$-invertible ideal of $R$.
\end{enumerate}
\end{lemma}
\begin{proof} Statement (1) is well-known (see \cite[Proposition 1.11]{gh2}), (2) is part of \cite[Corollary 1.7]{fg}), and (3) follows from \cite[Theorems 2.6, 2.18]{gh2} (but the ideas are from \cite{fg}). For (4), see, e.g. \cite[Theorem 1.9]{gh2}, and for (5), see \cite[Theorem 2.18 and Proposition 2.20]{gh2}.
\end{proof}

\begin{theorem} \label{t:potpullback} Consider the pullback diagram above. Then $R$ is $t$-potent if and only if each of the following conditions holds: \begin{enumerate}
\item $D$ is $t$-potent (or a field).
\item $N$ is $t$-potent for each $N \in$ $t$-{\rm Max($T$)} with $N \nsubseteq M$.
\item If $D$ is a field and $M$ is a $t$-ideal of $T$, then $M$ is $t$-potent in $T$.
\end{enumerate}
\end{theorem}
\begin{proof}  Suppose that $R$ is $t$-potent. 
If $D$ is not a field and $P' \in \tmx(D)$, then by Lemma~\ref{l:tmaxpullback}(3), $P:=\varphi^{-1}(P') \in \tmx(R)$, whence there is a $t$-rigid ideal   $I$ contained in $P$.  Then, again using Lemma~\ref{l:tmaxpullback}(3), it is easy to see that $\varphi(I)$ is a $t$-rigid ideal of $D$ contained in $P'$.   Hence $D$ is $t$-potent. Now let $N \in \tmx(T), N \nsubseteq M$.  Then (Lemma~\ref{l:tmaxpullback}(3)) $N \cap R \in \tm(R)$ and hence there is a $t$-rigid ideal $J$ contained in $N \cap R$. (In particular, $J \nsubseteq M$.) Then $JT$ is a $t$-rigid ideal of $T$ contained in $N$ (Lemma~\ref{l:tmaxpullback}(3)), and $N$ is $t$-potent. This holds whether $D$ is a field or not. If $D$ is a field, then $M$ is a maximal $t$-ideal and hence $t$-potent in $R$, and there is a $t$-rigid ideal $I$ of $R$ contained in $M$.  If $M$ is a (maximal) $t$-ideal of $T$, then $IT$ is a $t$-rigid ideal of $T$ contained in $M$, and hence $T$ is $t$-potent in this case.

For the converse, let $P \in \tmx(R)$.  If $P \supsetneq M$, then $P=\varphi^{-1}(P')$ for some $P' \in \tmx(D)$.  By assumption, there is a $t$-rigid ideal $C$ of $D$ contained in $P'$, and Lemma~\ref{l:tmaxpullback}(2,3) then implies that $\varphi^{-1}(C)$ is a $t$-rigid ideal of $R$ contained in $P$.

Next, suppose that $P=M$.  Then $D$ is a field.  By assumption $M$ is $t$-potent in $T$ or $M$ is not a $t$-ideal of $T$.  In the first case, let $A_1$ be a $t$-rigid ideal of $T$ contained in $M$.  In the second case, we have $M^{t_T}=T$  (where $t_T$ is the $t$-operation on $T$), and there is a finitely generated subideal $A_2$ of $M$ with $(A_2)^{t_T}=T$.  In either case, there is a finitely generated subideal $A$ of $M$ with $A \nsubseteq N$ for each $N \in \tmx(T) \setminus \{M\}$.  For such an $A$ we have $A=IT$ for some finitely generated ideal $I$ of $R$, and it is clear from the conditions satisfied by $A$ that $I$ is $t$-rigid in $R$.  

Finally, suppose that $P$ is incomparable to $M$.  Then $P=N \cap R$ for some $N \in \tmx(T)$ with $N \nsubseteq M$.  By assumption there is a $t$-rigid ideal $B$ of $T$ contained in $N$, and we may assume that $B$ contains an element $t \in T \setminus M$. Now $\varphi(t) \ne 0$, whence there is an element $t' \in T$ with $\varphi(tt')=1$.  This implies that $tt' \in R$, and, since $1-tt' \in M$, it is clear that $tt' \notin Q$ for each ideal $Q$ of $R$ such that $Q \supseteq M$. We consider three cases:

Case 1.  Suppose that $k$  is the quotient field of $D$.  Then $T$ is flat over $R$ (Lemma~\ref{l:tmaxpullback}), and hence $B=JT$ for some finitely generated ideal $J$ of $R$. 
By construction, $J$ is a $t$-rigid ideal of $R$ contained in $P= N \cap R$ in this case.

Case 2.  Suppose that $D$ is a field.  Then, arguing as in the ``$P=M$'' situation above, there is a finitely generated subideal $A$ of $M$ with $A \nsubseteq L$ for each $L \in \tmx(T) \setminus \{M\}$, and $A=IT$ for some finitely generated ideal $I$ of $R$.  Write $I=\sum_{i=1}^m Ra_i$ and $B=\sum_{j=1}^n Tb_j$, and let $J=\sum Ra_ib_j$.  Then $JT=IB \nsubseteq L$ for $L \in \tmx(T) \setminus \{M, N\}$.  It then follows easily that $J+Rtt'$ is a $t$-rigid ideal of $R$ contained in $P=N\cap R$.

Case 3. Suppose that $k$ is not the quotient field of $D$ and that $D$ is not a field, and put $S:=\varphi^{-1}(F)$, where $F$ is the quotient field of $D$.  By what has already been proved, $S$ is $t$-potent.  It then follows that $P=N \cap R$ is $t$-potent by Case 1 above.  This completes the proof.
\end{proof}

We next give an example, promised immediately after Proposition~\ref{p:increase}, of finite-type star operations $\st_1 \le \st_2$ on a domain $R$ and a $\st_2$-potent maximal $\st_2$-ideal that is $\st_1$-maximal but not $\st_1$-potent.

\begin{example} \label{e:increase} Let $F \subsetneq k$ be fields, $T=k[x_1, x_2, \ldots]$ a polynomial ring in countably many variables, and $R=F+M$, where $M$ is the maximal ideal of $T$ generated by the $x_i$.  In $R$, $M$ is both a maximal ($d$-)ideal and a maximal $t$-ideal.  Since $T$ is a Krull domain, it is $t$-potent, whence so is $R$ by Theorem~\ref{t:potpullback}.  In particular, $M$ is $t$-potent in $R$.  However, it is clear that each finitely generated subideal of $M$ is contained in infinitely many maximal ideals of ($T$ and hence of) $R$, and so $M$ is not $d$-potent.
\end{example}

\begin{theorem} \label{t:sppullback} Consider the pullback diagram at the beginning of this section. Then $R$ is $t$-super potent if and only if  $D$ is $t$-super potent and not a field, and each maximal $t$-ideal of $T$ not contained in $M$ is $t$-super potent.
\end{theorem}
\begin{proof} Assume that $R$ is $t$-super potent, and suppose, by way of contradiction, that $D$ is a field.  Then we have the following associated pullback diagram 

$$  \begin{CD}
        R_M   @>>>    D\\
        @VVV        @VVV    \\
        T_M  @>>>   k\\
\end{CD}$$
\smallskip

Choose $a \in M$, $a \ne 0$, and let $t \in T_M \setminus R_M$.  Then $at \notin aR_M$ since $t \notin R_M$, and $a \notin atR_M$, since $t^{-1} \notin R_M$.  However, this implies that $aR_M+atR_M$ is not principal, and hence that $aR_M$ is not $d$-super rigid. It follows that $MR_M$ is not $d$-super potent, and then, by Theorem~\ref{t:spchar}, that $M$ is not $t$-super potent, the desired contradiction. Thus $D$ is not a field.  

In the rest of the proof, we freely use Lemma~\ref{l:tmaxpullback}. Let $P'$ be a maximal $t$-ideal of $D$.  Then $P:=\varphi^{-1}(P')$ is a maximal $t$-ideal of $R$ properly containing $M$ and therefore contains a $t$-super rigid ideal $I$.  It is clear that $I':=\varphi(I)$ is contained in $P'$ and in no other maximal $t$-ideal of $D$.  Let $J' \supseteq I'$ be a finitely generated ideal of $D$.  Then, since $P$ is $t$-super potent, $J:=\varphi^{-1}(J')$ is a $t$-invertible ideal of $R$, and hence $\varphi(J)=J'$ is $t$-invertible in $D$.  Therefore, $D$ is $t$-super potent.
  
Now let $N \nsubseteq M$ be a maximal $t$-ideal of $T$. Then $N \cap R$ is a maximal $t$-ideal of $R$, and hence $T_N=R_{N \cap R}$ is $d$-super potent by Theorem~\ref{t:spchar}.  Therefore, since $N$ is $t$-potent by Theorem~\ref{t:potpullback}, $N$ is $t$-super potent by Theorem~\ref{t:spchar}.

For the converse, let $P \in \tmx(R)$.  If $P \supsetneq M$, then $\varphi(P)$ is $t$-super potent in $D$, and we can argue more or less as above to see that $P$ is $t$-super potent in $R$. 
Since $D$ is not a field, the only other possibility is $P=N \cap R$, where $N \in \tmx(T)$, $N \nsubseteq M$.  In this case, $t$-super potency of $N$ in $T$ yields $d$-super potency of $NT_N=(N \cap R)R_N$ (Theorem~\ref{t:spchar}).  Since $N \cap R$ is $t$-potent by Theorem~\ref{t:potpullback}, we may again apply Theorem~\ref{t:spchar} to conclude that $P=N \cap R$ is $t$-super potent.
\end{proof}

From Theorems~\ref{t:potpullback} and \ref{t:sppullback}, we can determine $t$-(super) potency in a large class of domains that appear frequently in the literature:

\begin{corollary} \label{c:psp} Let $D$ be a subdomain of the field $k$ and $x$ an indeterminate.  Let $R=D+xk[x]$ or $D+xk[[x]]$.  Then \begin{enumerate}
\item $R$ is $t$-potent if and only if $D$ is $t$-potent (or a field).
\item $R$ is $t$-super potent if and only if $D$ is $t$-super potent and not a field.
\end{enumerate}
\end{corollary}

Using Theorem~\ref{t:sppullback}, it is easy to give examples of $t$-super potent domains with non-$t$-super potent localizations: \bs

\begin{example} \label{e:loc} In the notation of Theorem~\ref{t:sppullback}, assume that $R$ is $t$-super potent. \begin{enumerate}
\item If the quotient field of $D$ is $F \ne k$, then $R_M$ is not $t$-super potent.  We may take $R$ integrally closed or not.
\item If $T$ is a one-dimensional local non-valuation domain, then $R_M$ is not $t$-super potent.
\end{enumerate}
\end{example}
\begin{proof} (1) In this case, let $S=\varphi^{-1}(F)$.  Then we have the pullback diagram 

$$  \begin{CD}
        R_M   @>>>    F\\
        @VVV        @VVV    \\
        T_M  @>>>   k,\\
\end{CD}$$
\smallskip

\noindent whence $R_M$ is not $t$-super potent by Theorem~\ref{t:sppullback}.  Let $x$ be an indeterminate, and $z$ an element of a field $k \supseteq \mathbb Q$.  Let $D= \mathbb Z$ and $T=\mathbb Q(z)[[x]]$ (so that $R=\mathbb Z +x\mathbb Q(z)[[x]]$).  In this case, if $z$ is an indeterminate, then $R$ is integrally closed.  On the other hand, if $z=\sqrt 2$, then $R$ is not integrally closed.

(2) If $k$ is not the quotient field of $D$, this follows from (1) .  If $k$ is the quotient field of $D$, then $R_M=T$ is not $t$-super potent by Corollary~\ref{c:locval}.
\end{proof}


\section{$t$-dimension one} \label{s:tdim1}

The primary goal of this section is to characterize generalized Krull domains using $t$-super potency.  We recall some definitions. First, a set $\mc P$ of prime ideals in a domain $R$ is a \emph{defining family} if $R=\bigcap_{P \in \mc P} R_P$. A defining family has \emph{finite character} (or is \emph{locally finite}) if each nonzero element $a \in R$ lies in at most finitely many elements of $\mc P$.  (Thus, in this terminology, if $\st$ is a finite-type star operation on $R$, then $R$ has finite $\st$-character if the defining family of maximal $\st$-ideals of $R$ has finite character.)  A prime $P$ of $R$ is \emph{essential} if $R_P$ is a valuation domain, and $R$ itself is an \emph{essential domain} if it possess a defining family of essential primes.  Finally, $R$ is a \emph{generalized Krull domain} if $R$ possesses a finite character defining family of height-one essential primes. For convenience we begin with a lemma, much of which comes from \cite{a} (and no doubt all of which is well known).

\begin{lemma} \label{l:spectral} Let $R$ be a domain and $\mc P$ a defining family for $R$.  Define $\st$ by $A^{\st}= \bigcap_{P \in \mc P} AR_P$ for each nonzero fractional ideal $A$ of $R$. Then: \begin{enumerate}
\item $\st$ is a star operation on $R$.
\item If $I$ is an integral ideal of $R$ for which $I^{\st} \ne R$, then $I \subseteq P$ for some $P \in \mc P$.
\item $P^{\st}=P$ for each $P \in \mc P$.
\item If $\mc P$ has finite character, then $\st$ has finite type.
\item If $\st$ has finite type, then: \begin{enumerate}
\item For each $P \in \mc P$, there is a maximal element $Q$ of $\mc P$ such that $P \subseteq Q$. Hence if $\mc P'$ denotes the set of maximal elements in $\mc P$, then $A^{\st}=\bigcap_{P \in \mc P'} AR_P$ for each nonzero fractional ideal $A$ of $R$.
\item Each proper $t$-ideal of $R$ is contained in some $P \in \mc P$.
\item $\st$-{\rm Max($R$)} $=\mc P'$.
\item If $\hgt(P)=1$ for each $P \in \mc P$ and $\mc Q$ denotes the set of height-one primes of $R$, then $\mc P =$ $t$-{\rm Max($R$)} $=$ $\st$-{\rm Max($R$)} $= \mc Q$.
\end{enumerate}
\end{enumerate}
\end{lemma}
\begin{proof}  Statements (1, 2, 3, 4) are in \cite{a}.  For (5a), Zorn's lemma applies since the union $P$ of a chain of elements of $\mc P$ satisfies $P^{\st} = P$ and by (2) $P \subseteq Q$ for some $Q \in \mc P$.  The ``hence'' statement follows easily.  
Statement (5b) follows from (2) in view of the fact that a proper $t$-ideal is also a proper $\st$-ideal. For (5c), if $Q \in \st\text{-Max}(R)$, then $Q \subseteq P$ for some $P \in \mc P'$ by (2).  But then $Q=P'$ by (3).  Hence $\st\text{-Max}(R) \subseteq P'$.  The reverse inclusion is trivial.  Finally, (5d) follows easily from (5a,b,c) and the fact that height-one primes are $t$-primes.
\end{proof}

\begin{remark} \label{r:max} With the notation of Lemma~\ref{l:spectral}, let $R$ be an almost Dedekind domain with exactly one non-invertible maximal ideal $M$, and let $\mc P$ denote the set of maximal ideals other than $M$.  Then, as is well known, $\mc P$ is a defining family for $R$, but the associated star operation does not have finite type. Indeed, conclusions (5b,d) fail to hold in this case: $M$ is a $t$-ideal but $M \nsubseteq P$ for all $P \in \mc P$.
\end{remark}

For our next result, recall that if $\st$ is a finite-type star operation on a domain $R$, then $R$ is said to have \emph{$\st$-dimension one} if each maximal $\st$-ideal of $R$ has height one.

\begin{theorem} \label{t:potent} Let $R$ be a domain and $\st$ a finite-type star operation on $R$.  Assume that $R$ has $\st$-dimension one and that $R$ is $\st$-potent. Then $R$ has finite $\st$-character.
\end{theorem}
\begin{proof} Denote the set of maximal $\st$-ideals of $R$ by $\{M_{\gamma}\}_{\gamma \in \Gamma}$.  For each $\gamma$, choose a $\st$-rigid ideal $I_{\gamma}$ contained in $M_{\gamma}$. Now suppose, by way of contradiction, that $a$ is a nonzero element of $R$ and $\Lambda$ is an infinite subset of $\Gamma$ with $a \in M_{\lambda}$ for $\lambda \in \Lambda$ and $a \notin M_{\gamma}$ for $\gamma \in \Gamma \setminus \Lambda$.  For $\lambda \in \Lambda$, $R_{M_{\lambda}}$
 is one-dimensional, and hence there is an element $s_{\lambda} \in R \setminus M_{\lambda}$ and a positive integer $n_{\lambda}$ for which $s_{\lambda}I_{\lambda}^{n_{\lambda}} \subseteq (a)$.  By construction, $(a,\{s_{\lambda}\})^{\st}=R$, whence $(a,s_1, \ldots, s_k)^{\st}=R$ for some finite subset $\{1, \ldots, k\}$ of $\Lambda$.  On the other hand, $(a,s_1, \ldots, s_k)I_1^{n_1} \cdots I_k^{n_k} \subseteq (a)$, and since $\Lambda$ is infinite, there is a maximal $\st$-ideal $M \in \{M_{\lambda}\}_{\lambda \in \Lambda} \setminus \{M_1, \ldots, M_k\}$ with $a \in M$.  However, $(a,s_1, \ldots, s_k) \nsubseteq M$ (since $(a,s_1, \ldots, s_k)^{\st}=R$ and) $I_j \nsubseteq M$ for $j=1, \ldots,k$, the desired contradiction.
 \end{proof}
 
The assumption on the $\st$-dimension in Theorem~\ref{t:potent} is necessary; for example, the Pr\"ufer domain $\mathbb Z+X\mathbb Q[[X]]$ is $d$-(super) potent but does not have finite $d$-character (note that $d=t$ here).

Theorem~\ref{t:potent} is, at first glance, a generalization of part of \cite[Corolllary 1.7]{acz}, which states that a $t$-potent domain of $t$-dimension one has finite $t$-character.  In fact, Theorem~\ref{t:potent} actually follows from \cite[Corollary 1.7]{acz}.  Indeed, if $R$ is as in Theorem~\ref{t:potent}, then Lemma~\ref{l:spectral} shows that $t$-{\rm Max($R$)} $= \st$-{\rm Max($R$)}.  (However, it is not generally the case that $\st=t$.) 
We have included the proof given above, since it seems much more conceptual than the one given in \cite{acz}.

As mentioned in the paragraph following Proposition~\ref{p:increase}, it is possible to have finite-type star operations $\st_1 \le \st_2$ on a domain $R$ with $R$ $\st_1$-potent but not $\st_2$-potent.  Indeed, this phenomenon can occur in a 2-dimensional P$v$MD.  We are grateful to the referee for suggesting the following construction.

\begin{example} \label{e:not potent} Let $T$ be the absolute integral closure of $\mathbb Z[X]$.  Since $\mathbb Z[X]$ is a Krull domain, $T$ is a P$v$MD, as was shown by H. Pr\"ufer \cite{pr} (see also the more recent paper by F. Lucius \cite{l}).  Moreover, it follows (Krull \cite[Satz 9]{kr}) that, since $\mathbb Z[X]$ has $t$-dimension one, so does $T$.  Now let $R$ be the localization of $T$ at a maximal ideal lying over $(2,X)$ in $\mathbb Z[X]$. Then $R$ is a (local and hence) $t$-potent domain of $t$-dimension one.  However, $R$ does not have finite $t$-character (since the ring of algebraic integers does not have finite character \cite[Proposition 42.8]{g2}), and hence $R$ is not $t$-potent by Theorem~\ref{t:potent} (or \cite[Corollary 1.7]{acz}).
\end{example}


In \cite{g}, Gilmer introduced the notion of sharpness.  The definition amounts to the following.  Call a maximal ideal $M$ of a domain $R$ \emph{sharp} if $\bigcap \{R_N \mid N \in \rm{Max}(R), N \ne M\} \nsubseteq R_M$, and call $R$ \emph{sharp} if each maximal ideal of $R$ is sharp.   In \cite{g} Gilmer focussed  on one-dimensional domains and proved that a sharp almost Dedekind is a Dedekind domain. (In a later paper, Gilmer and Heinzer \cite{gh} extended the ideas to higher dimensions, primarily in the setting of Pr\"ufer domains.)  The notion of sharpness was extended to star operations $\st$ of finite type in \cite[Remark 1.4]{ghl}: a maximal $\st$-ideal $M$ of a domain $R$ is $\st$-sharp if $\bigcap R_N \nsubseteq R_M$, where the intersection is taken over all maximal $\st$-ideals $N \ne M$.  Hence, for our purposes it is convenient to relabel ``sharp'' as ''$d$-sharp.''  It is relatively easy to prove that a $t$-potent maximal $t$-ideal must be $t$-sharp (see below), but this cannot be extended to arbitrary finite-type star operations.  In particular, it is not true for the $d$-operation, as can be seen by observing that maximal ideals of $k[x,y]$ are $d$-potent (as are the maximal ideals of any Noetherian domain) but are not $d$-sharp: if $M$ is maximal in $R:=k[x,y]$ and $u \in \bigcap \{R_N \mid N \in \mx(R), N \ne M\}$, then we must have $u \in R$, lest $(R:_R Ru)$ be contained in a height one prime and hence in infinitely many maximal ideals. 

\begin{proposition} \label{p:sharp} Let $R$ be a domain.  \begin{enumerate}
\item If $M$ is a $t$-potent maximal $t$-ideal of $R$, then $M$ is $t$-sharp.
\item If $\st$ is a finite-type star operation on $R$ and $M$ is a $\st$-super potent maximal $\st$-ideal of $R$, then $M$ is $\st$-sharp.
\end{enumerate}
\end{proposition}
\begin{proof} (1) This follows easily from \cite[proof of Theorem 1.2]{ghl}.  Here is a direct proof: Choose a $t$-rigid ideal of $R$ contained in $M$.  Since $I^v =I^t \subseteq M$, $I^{-1} \ne R$.  Choose $u \in I^{-1} \setminus R$.  Then $I \subseteq (R:_R Ru)$, and hence $(R:_R Ru) \nsubseteq N$ for each $N \in \tmx(R)$ with $N \ne M$.  On the other hand, since $u \notin R$, we must have $(R:_R Ru) \subseteq M$.  Hence $u \in \bigcap \{R_N \mid N \in \tmx(R), N \ne M\} \setminus R_M$, as desired.

(2) Let $\st$ be a finite-type star operation on $R$, $M$ be a $\st$-super potent maximal $\st$-ideal, and $I$ a $\st$-super rigid ideal contained in $M$. Since $M$ is also a (maximal) $t$-ideal (Theorem~\ref{t:increase2}), we have $I^{-1} \ne R$.  Then, as in the proof of (1), if we choose $u \in I^{-1} \setminus R$, then $u \in \bigcap \{R_N \mid N \in \textrm{$\st$-Max}(R), N \ne M\} \setminus R_M$.
\end{proof}

We observe that a $t$-sharp maximal $t$-ideal need not be $t$-potent (\cite[Example 1.5]{ghl}).  To force $t$-sharpness to imply $t$-potency, we add a finiteness condition. Recall that a fractional $t$-ideal $I$ of a domain $R$ has \emph{finite type} if $I=J^v$ for some finitely generated fractional ideal $J$.  We then say that $R$ is \emph{$v$-coherent} if $I^{-1}$ has finite type for each finitely generated fractional ideal $I$ of $R$.  (The notion of $v$-coherence, with a different name, was introduced by El Abidine \cite{e}.)  We then have from \cite[Theorem 1.6]{ghl} that a $v$-coherent $t$-sharp domain is $t$-potent.  The next result is immediate.

\begin{corollary} \label{c:sharp} A $v$-coherent $t$-sharp domain of $t$-dimension one has finite $t$-character. \qed
\end{corollary}

The above-mentioned theorem of Gilmer follows easily:

\begin{corollary} {\cite[Theorem 3]{g}} \label{c:gilmer} Let $R$ be a $d$-sharp almost Dedekind domain.  Then $R$ is a Dedekind domain.
\end{corollary}
\begin{proof} Any Pr\"ufer domain is $v$-coherent.  Moreover, the $d$- and $t$-operations coincide in a Pr\"ufer domain.  Hence $R$ has finite character by Corollary~\ref{c:sharp}, and it is well known that this implies that $R$ is a Dedekind domain.
\end{proof}

We now turn to the characterization of generalized Krull domains.  Since these domains are completely integrally closed, the next result will prove useful.  (Recall that a domain $R$ with quotient field $K$ is \emph{completely integrally closed} if, whenever $a \in R$ and $u \in K$ are such that $au^n \in R$ for each positive integer $n$, then $u \in R$.)

\begin{lemma} \label{l:cic} Let $R$ be a completely integrally closed domain and $M$ a $t$-super potent maximal $t$-ideal of $R$.  Then $\hgt(M)=1$.
\end{lemma}
\begin{proof} We proceed contrapositively.  Suppose that $M$ is a $t$-super potent maximal $t$-ideal of $R$ and that $P$ is a nonzero prime properly contained in $M$.  Choose a $t$-super rigid ideal $I \subseteq M$ with  $I \nsubseteq P$.  By Theorem~\ref{t:observations}, $P \subseteq \bigcap (I^n)^{\st}$, and hence $\bigcap (I^n)^{\st} \ne (0)$.  Therefore, $R$ is not completely integrally closed by \cite[Corollary 3.4]{aaz}.
\end{proof}

Recall that a domain $R$ is a \emph{Pr\"ufer $v$-multiplication domain} (P$v$MD) if each nonzero finitely generated ideal of $R$ is $t$-invertible; it is well known that $R$ is a P$v$MD if and only if each maximal $t$-ideal of $R$ is essential (note that the set of maximal $t$-ideals is always a defining family). 

\begin{theorem} \label{t:genkrull2} The following statements are equivalent for a domain $R$. \begin{enumerate}
\item $R$ is a generalized Krull domain.
\item $R$ is a $t$-potent essential domain of $t$-dimension one.
\item $R$ is a $t$-potent P$v$MD of $t$-dimension one.
\item $R$ is a completely integrally closed $t$-super potent domain.
\item $R$ is a $t$-super potent domain of $t$-dimension one.
\end{enumerate}
\end{theorem}
\begin{proof} (1) $\ra$ (2): Assume (1), and let $\mc P$ be a finite character defining family of height-one essential primes.  By Lemma~\ref{l:spectral}, $\mc P$ is in fact the set of maximal $t$-ideals of $R$.  (This also follows from \cite[Corollary 43.9]{g}). Hence $R$ has $t$-dimension one.  Also, $R$ is $t$-potent since $\mc P$ has finite character.

(2) $\ra$ (3): Assume (2), and let $\mc P$ be a defining family of essential primes.  For $P \in \mc P$, $PR_P$ is a $t$-prime in the valuation domain $R_P$, and it is well known (see, e.g., \cite[Lemma 3.17]{k}) that this implies that $P$ is a $t$-prime of $R$.  Then, since $R$ has finite $t$-character by Theorem~\ref{t:potent}, $\mc P$ also has finite character and is therefore the entire set of $t$-primes (Lemma~\ref{l:spectral}).  Therefore, $R_P$ is a valuation domain for each $t$-prime $P$, that is, $R$ is a P$v$MD.

(3) $\ra$ (4): Let $R$ be a $t$-potent P$v$MD of $t$-dimension one.  Then $R_P$ is a rank-one valuation domain for each $t$-prime $P$, and hence $R=\bigcap R_P$ is completely integrally closed.  Also, since $R$ is $t$-potent and $t$-locally $d$-super potent, $R$ is $t$-super potent by Theorem~\ref{t:spchar}.

(4) $\ra$ (5): This follows from Lemma~\ref{l:cic}.

(5) $\ra$ (1): Assume (5).  Then $R$ has finite $t$-character by Theorem~\ref{t:potent}, and $R_M$ is a valuation domain for each maximal $t$-ideal $M$ by Corollary~\ref{c:locval}.  Hence $R$ is a generalized Krull domain.
\end{proof}

One upshot of Theorem~\ref{t:genkrull2} is that a $t$-super potent domain of $t$-dimension one must be completely integrally closed.  Note that without the restriction on the $t$-dimension, a $t$-super potent domain need not even be integrally closed (Example~\ref{e:loc}).

We close with a brief discussion regarding the connection between P$v$MDs and $t$-super potent domains.  Observe that a $t$-potent P$v$MD is automatically $t$-super potent, but a P$v$MD need not be $t$-potent. (For example, a non-Dedekind almost Dedekind domain is a P$v$MD but is not $t$-potent (note that $d=t$ in this situation).)  In a P$v$MD, all nonzero finitely generated ideals are $t$-invertible, while in a $t$-super potent domain one has $t$-invertibility only ``above'' $t$-super rigid ideals.  Since, as is well known, if $I$ is a $t$-invertible ideal in a domain $R$, then both $I$ and $I^{-1}$ have finite type, a natural question arises: if $R$ is both $t$-super potent and $v$-coherent, must $R$ be a P$v$MD?  Even if we add the condition that $R$ be integrally closed, the answer is ``no,'' as is shown by the ring $R:=\mathbb Z + x\mathbb Q(z)[[x]]$ of Example~\ref{e:loc} (with $z$ an indeterminate):  $R$ is $t$-super potent by Theorem~\ref{t:sppullback}, is $v$-coherent by \cite[Theorem 3.5]{gh}, is integrally closed by standard pullback results, but is not a P$v$MD \cite[Theorem 4.1]{fg}.

We thank the referee for numerous suggestions that have greatly improved this paper.

\end{document}